\documentclass[oneside,a4paper,12pt,notitlepage]{article}
\usepackage[left=0.8in, right=0.8in, top=1.5in, bottom=1.5in]{geometry}

\usepackage[T1]{fontenc} 
\usepackage[utf8]{inputenc} 
\usepackage[english]{babel}
\usepackage{lipsum} 
\usepackage{lmodern}
\usepackage{amssymb}
\usepackage{amsthm}
\usepackage{bm}
\usepackage{mathtools}
\usepackage{braket}
\usepackage{esint}

\usepackage{enumitem}
\usepackage{booktabs}
\usepackage{graphicx}
\usepackage{tikz}
\usetikzlibrary{patterns}
\usepackage{multicol}
\usepackage{caption}
\usepackage[skins,theorems]{tcolorbox}
\tcbset{highlight math style={enhanced,
		colframe=black,colback=white,arc=0pt,boxrule=1pt}}
\def\XXint#1#2#3{{\setbox0=\hbox{$#1{#2#3}{\int}$}
    \vcenter{\hbox{$#2#3$}}\kern-.5\wd0}}

\theoremstyle{definition}
\newtheorem{definizione}{Definition}[section]
\theoremstyle{plain}
\newtheorem{teorema}{Theorem}[section]

\newtheorem{lemma}[teorema]{Lemma}

\newtheorem{corollario}[teorema]{Corollary}
\theoremstyle{definition}
\newtheorem{esempio}{Example}[section]
\newtheorem{oss}[esempio]{Remark}

\DeclareMathOperator{\R}{\mathbb{R}}

\makeatletter
\newcommand{\myfootnote}[2]{\begingroup
	\def\@makefnmark{}%
	\addtocounter{footnote}{-1}%
	\footnote{\textbf{#1} #2}
	\endgroup}
\makeatother

\usepackage{hyperref}
\hypersetup{linktoc=none, bookmarksnumbered, colorlinks=true, linkcolor=red}

 \title{On a Serrin-type overdetermined problem}
\author{Celentano A., Nitsch C., Trombetti C.}
\date{}

\newcommand{\Addresses}{{
  \bigskip 
  \footnotesize 
 
  \medskip 
 
  \noindent\textit{E-mail address}, A.~Celentano: \texttt{antonio.celentano2@unina.it}\\ 
    \noindent\textit{E-mail address}, C.~Nitsch: \texttt{c.nitsch@unina.it}\\
    \noindent\textit{E-mail address}, C.~Trombetti: \texttt{cristina@unina.it}
   \medskip 
     
    \noindent\textsc{Dipartimento di Matematica e Applicazioni ``R. Caccioppoli'', Universit\`a degli studi di Napoli Federico II, Via Cintia, Complesso Universitario Monte S. Angelo, 80126 Napoli, Italy.}

    \par\nopagebreak 

}}

\begin{document}
\maketitle

\begin{abstract} 
\noindent In this paper, we prove a Serrin-type result for an elliptic system of equations, overdetermined with both Dirichlet and a generalized Neumann conditions. With this tool, we characterize the critical shapes under volume constraint of some domain functionals.   

 

\end{abstract}
\Addresses 

\section{Introduction}
In the study of overdetermined elliptic problems, Serrin's pioneering work \cite{Serrin} marked a turning point. He proved the radial symmetry property of the solution to torsion problem, overdetermined with a Neumann boundary condition. More precisely, main result of \cite{Serrin} states the following:
\begin{teorema}\label{serr}
Let $\Omega\subset\mathbb{R}^n$ be a connected, bounded and open set whose boundary is of class $C^2$. If there exists a solution $u\in C^2(\overline{\Omega})$ satisfying 
\begin{equation}\label{sp}
\begin{cases}
-\Delta u=1 & \mbox{in } \Omega,\\ 
u=0& \mbox{on } \partial\Omega,\\
\frac{\partial u}{\partial \nu}=c& \mbox{on } \partial\Omega,\\
\end{cases}
\end{equation}
where $\nu$ is the inward normal to $\partial\Omega$. Then $\Omega=B_R(x_0)$ is a ball of radius $R$, centered at $x_0$ and $u(x)=\frac{R^2-|x-x_0|^2}{2n}$.
\end{teorema}
\noindent As remarked in \cite{Serrin}, Theorem \ref{serr} still holds if Poisson equation in \eqref{sp} is replaced by a non linear elliptic equation, whose structure satisfies some suitable conditions, provided the solution $u$ has constant sign in $\Omega$. For instance, Poisson semilinear equation $-\Delta u=f(u)$ can be taken into account. 

\medskip

\noindent The proof relies on two main tools: the moving planes method and a refinement of the Hopf Boundary Lemma \cite[Lemma 1]{Serrin}. The former was initially introduced by Alexandrov in \cite{alexandrov} to prove that the only compact, embedded $(n-1)$-dimensional $C^2$ hypersurfaces in $\mathbb{R}^n$ with constant mean curvature are the spheres. Serrin recovered this technique and adapted it to the PDEs framework. Since then, many generalization of this method have been made. We mention just a few of these. In \cite{GNN}, it was extended and applied to show symmetry and monotonicity properties of positive solutions to a class of elliptic problems with Dirichlet boundary condition. Afterwards, it was improved in order to treat more general domains with weaker regularity assumptions on the solution, see \cite{bereniren, dancer}. The moving planes method can be applied successfully
also in the case of $p$-Laplace operator \cite{DP, DS}. For a deep analysis involving also a quantitative approach, we refer to \cite{ciron}. 

\medskip

\noindent We point out that many different proofs of Theorem \ref{serr} and further extensions are currently available in the literature, see \cite{NT} for a survey on this topic. Weinberger provided an alternative short proof in \cite{weinberger}, based on the use of a suitable auxiliary function, called $P$-function, and Poho\v{z}aev identity. These techniques were extended to handle a general quasilinear operator in \cite{fk}. An other approach, which combines Poho\v{z}aev identity together with Newton inequalities, is shown in \cite{bnst}. With these tools, authors are also able to deal with a class of fully non linear equations, namely Hessian equations. We mention also the work \cite{Chenr}, where a deep connection between shape optimization and overdetermined problems is emphasized. Authors construct a shape functional  minimized by solutions to Serrin's problem. They infer the uniqueness of minimizer and then, expressing the optimality condition by means of the domain derivative, they show that the mean curvature of the boundary of the minimizer must be constant. Hence, it must be a ball.

\medskip

\noindent In this paper, we deal with an extension of the Serrin's problem to system of elliptic equations with Dirichlet datus, overdetermined with an other boundary condition which involves the normal derivatives of solutions and generalizes the Neumann one. Constraints of this type arise from shape optimization context, by imposing the optimality condition computed via domain derivative, see for instance \cite{Bmmtv}. The main result of this work is the following:
\begin{teorema}\label{main}
Let $\Omega\subset\R^n$ be a bounded, open, connected set of class $C^2$. For $i=1,..., m$, let $u_i\in C^2(\overline{\Omega})$, $u_i>0$ in $\Omega$, be solutions to
\begin{equation}
\begin{cases}
-\Delta u_i=f_i(u_i) & \mbox{in } \Omega,\\ 
u_i=0 & \mbox{on } \partial\Omega,\\
F\left(\frac{\partial u_1}{\partial \nu},..., \frac{\partial u_m}{\partial \nu}\right)=c& \mbox{on } \partial\Omega,\label{neum}\\
\end{cases}
\end{equation}
where $\nu$ is the inner unit normal to $\partial \Omega$ and $c\in \mathbb{R}$.  Let us assume that $f_i$ and $F$ satisfy the following assumptions:
 \begin{enumerate}[label=(\roman*)]
     \item $f_i=g_i+h_i$, with $g_i$ locally Lipschitz in $\R$ and $h_i$ non decreasing,
     \item $F:[0, +\infty)^m\rightarrow \mathbb{R}$ and $F\in C^1\left((0,+\infty)^m\right)$,
     \item $\forall \,i\in\{1,..., m\}$, $\forall \,(x_1,...,x_m)\in [0, +\infty)^m$, $\forall h>0$, it holds
     \begin{equation*}\label{monotonicity}
      F(x_1,...,x_i,...,x_m)\leq F(x_1,...,x_i+h,...,x_m)
     \end{equation*}
    \item $\exists\,j\in\{1,.., m\}$ such that $\frac{\partial F}{\partial x_j}>0$ in $(0, +\infty)^m$.
 \end{enumerate}
 Then $\Omega$ is a ball of radius $R$ and the solutions $u_i$ are radially symmetric. Moreover,
 \begin{equation*}
\dfrac{\partial u_i}{\partial r}<0 \quad \text{for }\{0<r<R\},\,\forall \,i\in\{1,..., m\}.
  \end{equation*}
\end{teorema}
\noindent The proof is based on moving planes method and the arguments used are inspired by those of Serrin. In the end, we provide an application of the previous result to the study of critical shapes of some domain functionals involving the torsional rigidity $T(\Omega)$ and the first eigenvalue $\lambda_1(\Omega)$ of Dirichlet-Laplacian.
\begin{corollario}\label{torlambda}
Let $\Omega\subset\mathbb{R}^n$ be an open, bounded, connected set of class $C^{2,\gamma}$ for $\gamma\in(0,1)$. If $\Omega$ is a critical shape under volume constraint for the functional $J(\Omega)=T^{\alpha}(\Omega)\lambda_1^{\beta}(\Omega)$, with $\alpha, \beta \in \mathbb{R}$ not both zero and such that $\alpha\beta\leq 0$, then it is a ball. 
\end{corollario}
\noindent The paper is organized as follows. In Section \ref{mmp}, we provide a detailed exposure of the moving planes method, recalling also some useful results proved in \cite{GNN}. In Section \ref{domder}, we give introductory notions and tools on the torsional rigidity, the first eigenvalue of Laplacian with Dirichlet boundary condition and domain derivative of shape functionals. Finally, in Section \ref{mt} we prove Theorem \ref{main} and Corollary \ref{torlambda}.
\section{Notions and preliminaries}
\subsection{The moving planes method}\label{mmp}
We provide a geometric description of the procedure of moving up planes perpendicular to a fixed direction, until a critical position is reached.
Let $\Omega\subset\R^n$ be a bounded domain with smooth boundary. Assigned a unit vector $\gamma$ in $\R^n$ and a real value $\lambda$, we define the hyperplane $T_{\lambda}=\{x\in\R^n\,|\,\gamma\cdot x=\lambda\}$, that is the boundary of the open halfspace $H_{\lambda}=\{x\in\R^n\,|\,\gamma\cdot x>\lambda\}$. If $\lambda$ is large enough, then $T_{\lambda}\cap \overline{\Omega}=\emptyset$. Decreasing $\lambda$, the hyperplane $T_{\lambda}$ moves continuously toward $\Omega$ along the direction $-\gamma$ and it begins to intersect $\overline{\Omega}$ for $\lambda=\lambda_0$. If $\lambda<\lambda_0$, the open cap $\Sigma(\lambda)=H_{\lambda}\cap\Omega$ is nonempty. We denote respectively by $\Sigma'(\lambda)$ and $x^\lambda$, the reflection of the open cap $\Sigma(\lambda)$  and of the point $x$ in the hyperplane $T_{\lambda}$. At the beginning, $\Sigma'(\lambda)$ will be in $\Omega$ and as $\lambda$ decreases, the reflected cap $\Sigma'(\lambda)$ will remain in $\Omega$, at least until one of the following two events occurs:  
\begin{enumerate}[label=(\roman*)]
     \item $\Sigma'(\lambda)$ becomes internally tangent to $\partial \Omega$ at some point $\overline{x}$ not on $T_{\lambda}$,
     \item $T_{\lambda}$ reaches a position such that it is orthogonal to $\partial\Omega$ at some point $\overline{y}$.
 \end{enumerate}
We denote by $T_{\lambda_c}$ the hyperplane $T_{\lambda}$ when it first reaches one of these positions. Now, we recall in a special form a result proved by Gidas, Ni and Nirenberg in \cite[Theorem 2.1]{GNN}.  
\begin{teorema}\label{gnn1}
Let $\Omega\subset\R^n$ be an open, connected set of class $C^2$. Let $u\in C^2(\overline{\Omega})$, $u>0$ in $\Omega$, be a solution to
\begin{equation}\label{semprob}
\begin{cases}
-\Delta u=f(u) & \mbox{in } \Omega,\\ 
u=0 & \mbox{on } \partial\Omega,\\
\end{cases}
\end{equation}
 with $f=f_1+f_2$, where $f_1$ is locally Lipschitz in $\R$ and $f_2$ is non decreasing. Then 
  \begin{equation}\label{in1}
      \nabla u(x)\cdot \gamma <0, \quad u(x)< u(x^{\lambda})  \quad \forall\,x\in\Sigma(\lambda),\quad \forall\,\lambda\in(\lambda_0, \lambda_c).
  \end{equation}
Thus $\nabla u\cdot \gamma <0$ in $\Sigma(\lambda_c)$. Furthermore if  $\nabla u\cdot \gamma=0$ at some point on $\Omega\cap T_{\lambda_c}$ then necessarily $u$ and $\Omega$ are symmetric in the hyperplane $T_{\lambda_c}$. 
\end{teorema}
\begin{oss} \label{nonlin}
As specified in \cite{GNN}, Theorem \ref{gnn1} is verified for a class of nonlinear elliptic second order equations. With $\Omega$ as before, let $u\in C^2(\overline{\Omega})$ be a solution to the following equation
\begin{equation}\label{nonlinear}{F(x, u, u_{x_1},..., u_{x_n}, u_{x_1 x_1},..., u_{x_n x_n})=0 \quad \text{in }\Omega,}\end{equation}
which is elliptic in $\Omega$ if there exist positive constants $m, M$ such that
\begin{equation*}{m|\xi|^2\leq \sum_{i,j=1}^n F_{u_{ij}}(x)\xi_i\xi_j\leq M|\xi|^2 \quad \forall\, \xi\in\mathbb{R}^n,\quad \forall\, x\in\Omega.}\end{equation*}
The function $F(x, z, p_k, r_{i,j})$ satisfies the following conditions:
\begin{enumerate}
\item [a)] $F$ is continuous and it has continuous first derivative with respect to $z, p_k$ and $r_{i,j}$ for all values of these arguments, for every $x\in \Omega$;
\item [b)] the function $g(x)=F(x, 0,..., 0)$ verifies
\begin{equation}{g(x)\geq0 \quad\forall\, x\in\partial\Omega\cap\{x_1>\lambda_c\},\quad\text{or}\quad g(x)<0 \quad\forall x\in\partial \Omega\cap\{x_1>\lambda_c\};\nonumber}\end{equation}
\item [c)] for every $\lambda$ such that $\lambda_c\leq \lambda<\lambda_0$, for $x\in\Sigma(\lambda)$ and all values of the arguments $z, p_k, r_{i,j}$ with $z>0$ e $p_1<0$, it results
\begin{equation*}{F(x^\lambda, z, -p_1, p_2,..., p_n, r_{1,1}, -r_{1,\alpha},..., r_{\beta,\gamma})\geq F(x, z, p_1, p_2,..., p_n, r_{1,1}, r_{i,j}),\label{Fsimm}}\end{equation*}
where $i, j=1,... n$ and $\alpha, \beta, \gamma=2,... n$.
\end{enumerate}
\end{oss}
\noindent As a consequence of Theorem \ref{gnn1}, the following symmetry result holds \cite[Theorem 1]{GNN}.
\begin{teorema}\label{gnn2}
Let $f:\mathbb{R}\to\mathbb{R}$ be such that $f=f_1+f_2$, where  $f_1$ is a locally Lipschitz function and $f_2$ is non decreasing. Then, any positive solution $u\in C^2(\overline{B})$ to the problem
    \begin{equation*}
\begin{cases}
    -\Delta u= f(u) & \text{in } B,\\
    u=0 &\text{on } \partial  B,
\end{cases}
\end{equation*}
where $B$ is a ball of $\mathbb{R}^n$ of radius $R$, has to be radial and 
\begin{equation*}
\dfrac{\partial u}{\partial r}<0,\quad\text{for }\,0<r<R.
\end{equation*}
\end{teorema}
\noindent In \cite{GNN}, Theorem \ref{gnn1} is also extended to some domains with corners, by proving a refinement of the Hopf Boundary Lemma \cite[Lemma S]{GNN} reported below, which is a generalization of a result due to Serrin in \cite[Lemma 1]{Serrin}.
\begin{lemma}\label{hopf}
Let $\Omega\subset\R^n$ be an open, connected set with the origin $O$ on its boundary. Assume that near $O$ the boundary consists of two trasversally intersecting $C^2$ hypersurfaces $\rho=0$ and $\sigma=0$. Suppose $\rho, \sigma<0$ in $\Omega$. Let $w\in C^2(\overline{\Omega})$, with $w<0$ in $\Omega$, $w(O)=0$, satisfying the differential inequality
\begin{equation*}
      Lu=\sum_{i,j=1}^n a_{ij}(x)\frac{\partial^2 u}{\partial x_i\partial x_j}(x)+\sum_{i=1}^nb_i(x)\frac{\partial u}{\partial x_i}(x)+c(x)u(x)\geq 0\quad \text{in }\Omega,
  \end{equation*}
  where $L$ is a uniformly elliptic operator and its coefficients are uniformly bounded. Assume
  \begin{equation}\label{cond1}
      \sum_{i,j=1}^n a_{ij}\rho_{x_i}\sigma_{x_j}\geq 0 \quad \text{at O}.
  \end{equation}
  If this is $0$, assume furthermore that $a_{ij}\in C^2$ in $\overline{\Omega}$ near $O$, and that
   \begin{equation*}
     \sum_{i,j=1}^nD(a_{ij}\rho_{x_i}\sigma_{x_j})=0 \quad \text{at O},
  \end{equation*}
  for any first order derivative $D$ at $O$ tangent to the submanifold $\{\rho=0\}\cap\{\sigma=0\}$. Then, for any direction $s$ at $O$ which enters $\Omega$ trasversally to each hypersurface
  \begin{equation*}
    \frac{\partial w}{\partial s}<0 \quad \text{at O in case of strict inequality in \eqref{cond1}},
  \end{equation*}
  \begin{equation*}
    \frac{\partial w}{\partial s}<0\quad \text{or}\quad \frac{\partial^2 w}{\partial s^2}<0\quad \text{at O in case of equality in \eqref{cond1}}.
  \end{equation*}
\end{lemma}
\subsection{The domain derivative}\label{domder}
Fixed a bounded and open set $\Omega\subset \mathbb{R}^n$, the torsional rigidity of $\Omega$ is
\begin{align*}
    T(\Omega)=\int_{\Omega}u_{\Omega}\,dx,
\end{align*}
where $u_{\Omega}$ is the torsion function of $\Omega$, that is the unique solution of the boundary value problem
\begin{equation*}
\begin{cases}
-\Delta u=1 & \mbox{in } \Omega,\\ 
u\in W^{1,2}_0(\Omega).\\
\end{cases}
\end{equation*}
Let $\lambda_1(\Omega)$ be the first eigenvalue of the Laplacian with Dirichlet boundary conditions. It is known that $\lambda_1(\Omega)>0$, we assume it is simple and we indicate by $v_{\Omega}$ the normalized corresponding eigenfunction, which is the solution to
\begin{equation*}
\begin{cases}
-\Delta v=\lambda_1(\Omega)v & \mbox{in } \Omega,\\ 
v\in W^{1,2}_0(\Omega),\\
\end{cases}
\end{equation*}
such that $\|v_{\Omega}\|_{L^2(\Omega)}=1$.\\

Next, we recall the first order shape derivative of the two functionals $T(\Omega), \lambda_1(\Omega)$.\\ Let $\theta\in W^{1, \infty}(\mathbb{R}^n, \mathbb{R}^n)$ be a Lipschitz continuous function. For given $t_0>0$, we denote by $\{\Omega_t\}_{0\leq t<t_0}$ a family of perturbations of the domain $\Omega$ of the form \begin{equation*} \Omega_t = (I+t\theta)(\Omega),
\end{equation*} 
where $I$ is the identity in $\mathbb{R}^n$. 
\begin{oss}
If $t_0$ is small enough, then $\Phi(t,x)=I(x)+t\theta(x)$ is a bi-Lipschitz homeomorphism for every $0\leq t<t_0$. Accordingly, $\Omega_t$ is measurable and open if and only if $\Omega$ is measurable and open. 
\end{oss}
\begin{definizione}
The derivative of a given functional $\mathcal{J}$ at $\Omega$ in the direction $\theta$ is defined as
\begin{equation}\label{shapeder} d\mathcal{J}(\Omega, \theta) := \lim_{t\to 0^+}\frac{\mathcal{J}(\Omega_t)-\mathcal{J}(\Omega)}{t},
\end{equation}
if the limit exists. 
\end{definizione}
\noindent For the next two results, we refer to \cite[Corollary 5.3.8, Theorem 5.7.1]{Henrot2}.
\begin{teorema}\label{torsion} Let $\Omega$ be an open, bounded and of class $C^1$ set. Then
\begin{equation*} dT(\Omega, \theta)=\int_{\partial\Omega}|\nabla u_{\Omega}|^2\theta\cdot n\,d\,\mathcal{H}^{n-1},
\end{equation*}
where $n$ denotes the outer unit normal to $\partial\Omega$.
\end{teorema}
\begin{teorema}\label{firsteigen}
Let $\Omega$ be an open, bounded and of class $C^2$ set. If $\lambda_1 (\Omega)$ is a simple eigenvalue, then 
\begin{equation*} d\lambda_1(\Omega, \theta)=-\int_{\partial\Omega}|\nabla v_{\Omega}|^2\theta\cdot n\,d\,\mathcal{H}^{n-1}.
\end{equation*}
\end{teorema}
\noindent Now, we consider the volume functional $V(\Omega)=|\Omega|$, which gives the Lebesgue measure of $\Omega$. For a perturbation $\Phi$ of the type above, it results \cite[Section 2.3]{Bandle1} 
\begin{equation*}dV(\Omega, \theta)=\int_{\Omega}\text{div}\,\theta\,dx.
\end{equation*}
\begin{definizione}
The perturbation $\Phi(t, x)$ is called volume preserving if 
\begin{equation}\label{volpres}
\int_{\Omega}\text{div}\,\theta\,dx=0.
\end{equation}
\end{definizione}
\begin{definizione}
A set $\Omega$ is said a critical shape under volume constraint for a functional $\mathcal{J}$ if \eqref{shapeder} vanishes for every choice of $\theta$ satisfying \eqref{volpres}.
\end{definizione}

\section{Proof of the main theorem}\label{mt}

\begin{proof}[Proof of Theorem \ref{main}]
Let fix a direction $\gamma$ and let apply the procedure described before, until the hyperplane $T_{\lambda}$ reaches a critical position. The goal is to prove that $\Omega$ is symmetric with respect to the hyperplane $T_{\lambda_c}$. Once this fact is proved then the statement follows, since for every direction $\gamma$, $\Omega$ would be symmetric with respect to the hyperplane normal to $\gamma$. Moreover, according to the construction, $\Omega$ would also be simply connected, then it has to be a ball. In the end, an application of Theorem \ref{gnn2} shows that $u_i$ is radially symmetric and non increasing, for all $i=1,..., m$.\\ We set $u_i^{\lambda_c}$ the maps defined in $\Sigma'(\lambda_c)$ by
\begin{equation*}
u_i^{\lambda_c}(x)=u_i(x^{\lambda_c})\quad x\in\Sigma'(\lambda_c),\quad \forall\, i\in\{1,..., m\}.
\end{equation*}
The functions $u_i^{\lambda_c}$ satisfy
\begin{equation*}
\begin{cases}
-\Delta u_i^{\lambda_c}=f_i(u_i^{\lambda_c}) & \mbox{in } \Sigma'(\lambda_c),\\
u_i=u_i^{\lambda_c}&\mbox{on }\partial\Sigma'(\lambda_c)\cap\partial H_{\lambda_c},\\  u_i^{\lambda_c}=0 & \mbox{on } \partial\Sigma'(\lambda_c)\setminus\partial H_{\lambda_c},\\
F\left(\frac{\partial u_1^{\lambda_c}}{\partial \nu},..., \frac{\partial u_m^{\lambda_c}}{\partial \nu}\right)=c& \mbox{on }\partial\Sigma'(\lambda_c)\setminus\partial H_{\lambda_c}.
\end{cases}
\end{equation*}
Since $\Sigma'(\lambda_c)$ is contained in $\Omega$, it's possible to consider $u_i-u_i^{\lambda_c}$. As a consequence of Theorem \ref{gnn1}, $u_i-u_i^{\lambda_c}\geq 0$ and it solves a linear differential inequality
\begin{equation*}
 0\geq h_i(u_i^{\lambda_c})-h_i(u_i)=\Delta (u_i-u_{\lambda_c})+g_i(u_i)-g_i(u_i^{\lambda_c})=\Delta (u_i-u_i^{\lambda_c})+c_i(u_i-u_i^{\lambda_c}),
\end{equation*}
where
\begin{equation*}{c_i(x)=\begin{cases}\dfrac{g_i(u_i(x))-g_i(u_i^{\lambda_c}(x))}{u_i(x)-u_i^{\lambda_c}(x)} & \mbox{if }u(x)\neq u_i^{\lambda_c}(x),\\ 0 & \mbox{if }u(x)=u_i^{\lambda_c}(x).\end{cases}\nonumber}\end{equation*}
An application of the strong maximum principle gives either
\begin{equation*}
u_i-u_i^{\lambda_c}>0\quad \text{or}\quad u_i-u_i^{\lambda_c}\equiv 0\quad \text{in }\Sigma'(\lambda_c).
\end{equation*}
The latter case would imply that $\Omega$ is symmetric about $T(\lambda_c)$. Suppose that case $i)$ occurs, that is $\Sigma'(\lambda_c)$ is internally tangent to $\partial\Omega$ at a point $\overline{x}$ not belonging to $T_{\lambda_c}$ and assume by contradiction that $u_i-u_i^{\lambda_c}>0$ in $\Sigma'(\lambda_c)$ for all $i=1,..., m$. Then Hopf Lemma ensures that 
\begin{equation*}
\frac{\partial (u_i-u_i^{\lambda_c})}{\partial \nu(\overline{x})}>0.
\end{equation*}
Hence, we get
\begin{equation*}
\frac{\partial u_i}{\partial \nu(\overline{x})}>\frac{\partial u_i^{\lambda_c}}{\partial \nu(\overline{x})}\geq 0.
\end{equation*}
Using the hypothesis on the function $F$, we get a contradiction
\begin{equation*}
c=F\left(\frac{\partial u_1^{\lambda_c}}{\partial \nu(\overline{x})},..., \frac{\partial u_m^{\lambda_c}}{\partial \nu(\overline{x})}\right)<F\left(\frac{\partial u_1}{\partial \nu(\overline{x})},..., \frac{\partial u_m}{\partial \nu(\overline{x})}\right)=c.
\end{equation*}
In the case $ii)$, the proof makes use of Lemma \ref{hopf}. The goal is to prove that $u_i-u_i^{\lambda_c}$ have in $\overline{y}$ a second order zero, for some $i\in\{1,..., m\}$. Now, we fix a coordinate system with the origin at $\overline{y}$, the $x_n$ axis in the direction of the inward normal to $\partial \Omega$ at $\overline{y}$ and the $x_1$ axis in the direction of $\gamma$, that is normal to $T_{\lambda_c}$. In this coordinates system, the boundary of $\Omega$ is locally described by 
\begin{equation*}
x_n=\phi(x_1, x_2,..., x_{n-1}),\quad \phi\in C^2.
\end{equation*}
By construction, the function $\phi$ satisfies the following properties \begin{equation}\label{phi}
\nabla\phi(0)=0,\quad \frac{\partial^2 \phi}{\partial x_1 \partial x_j}(0)=0\quad \text{for }j=2,..., n-1.
\end{equation}
Indeed, $\nu(0)=(-\nabla\phi(0), 1)=(0,..., 0, 1)$ which gives the former in \eqref{phi}. Moreover, $\frac{\partial \phi}{\partial x_1}$ has an extremum point at $0$ with respect to all but the first coordinates directions, since $\Sigma'(\lambda_c)\subseteq \Omega$. Let us set
\begin{equation*}
u_i^{\lambda_c}(x_1, x_2,..., x_n)=u_i(-x_1, x_2,..., x_n).
\end{equation*}
In order to prove that all the first and second derivatives of $u_i$ and $u_i^{\lambda_c}$ coincide at $0$, it is enough to show
\begin{equation*}
\frac{\partial u_i}{\partial x_1}(0)=0,\quad \frac{\partial^2 u_i}{\partial x_1 \partial x_j}(0)=0\quad\text{for }j=2,..., n.
\end{equation*}
Since $u_i\in C^2$, the boundary condition $u_i=0$ on $\partial\Omega$ can be written as
\begin{equation}\label{diri}
u_i(x_1, x_2,..., x_{n-1}, \phi)=0\quad \text{for }i=1,..., m.
\end{equation}
Differentiating \eqref{diri} with respect to $x_j$ for $j=1,..., n-1$, we have
\begin{equation}\label{diri2}
\frac{\partial u_i}{\partial x_j}+\frac{\partial u_i}{\partial x_n}\frac{\partial \phi}{\partial x_j}=0.
\end{equation}
Evaluating \eqref{diri2} at $0$ and for $j=1,..., n-1$, it results $\frac{\partial u_i}{\partial x_j}(0)=0$. 
Differentiating \eqref{diri2} with respect to $x_k$ for $k=1,..., n-1$, we get 
\begin{equation}\label{diri3}
    \frac{\partial^2 u_i}{\partial x_j\partial x_k}+\frac{\partial^2 u_i}{\partial x_j \partial x_n}\frac{\partial \phi}{\partial x_k}+\frac{\partial^2 u_i}{\partial x_n \partial x_k}\frac{\partial \phi}{\partial x_j}+\frac{\partial^2 u_i}{\partial x_n \partial x_n}\frac{\partial \phi}{\partial x_j}\frac{\partial \phi}{\partial x_k}+\frac{\partial u_i}{\partial x_n}\frac{\partial^2 \phi}{\partial x_j\partial x_k}=0.
\end{equation}
Evaluating \eqref{diri3} at $0$, for $j=1$ and $k=2,..., n-1$, it results $\frac{\partial^2 u_i}{\partial x_1 \partial x_k}(0)=0$. 
If it exists $j$ such that $\frac{\partial u_j}{\partial x_n}(0)=0$, then the normal derivative $\frac{\partial u_j}{\partial\nu}$ has a minimum in $0$. Since
\begin{equation*}\label{dernorm}\frac{\partial u_j}{\partial\nu}(x_1, x_2,..., x_{n-1}, \phi)=\nabla u_j(x_1,..., x_{n-1}, \phi)\cdot \frac{(-\nabla \phi(x_1,..., x_{n-1}), 1)}{\sqrt{1+|\nabla \phi(x_1,..., x_{n-1})|^2}},\end{equation*}
a straightforward calculation shows that
\begin{align}\label{derx1}&\frac{\partial}{\partial x_1}\left(\frac{\partial u_j}{\partial \nu}(x_1,..., x_{n-1}, \phi)\right)=\Biggl(\,\frac{\partial^2 u_j}{\partial x_n\partial x_1}+\frac{\partial^2 u_j}{\partial x_n\partial x_n}\frac{\partial \phi}{\partial x_1}-\sum_{k=1}^{n-1}\frac{\partial^2 u_j}{\partial x_k\partial x_1}\frac{\partial \phi}{\partial x_k}-\sum_{k=1}^{n-1}\frac{\partial^2 u_j}{\partial x_k\partial x_n}\frac{\partial \phi}{\partial x_1}\frac{\partial \phi}{\partial x_k}+\\\nonumber& -\sum_{k=1}^{n-1}\frac{\partial u_j}{\partial x_k}\frac{\partial^2 \phi}{\partial x_k\partial x_1}\Biggr)\frac{1}{\sqrt{1+|\nabla\phi|^2}}-\left( \frac{\partial u_j}{\partial x_n}-\sum_{k=1}^{n-1}\frac{\partial u_j}{\partial x_k}\frac{\partial \phi}{\partial x_k}\right)\sum_{k=1}^{n-1}\frac{\partial \phi}{\partial x_k}\frac{\partial^2 \phi}{\partial x_k\partial x_1}\left(1+|\nabla \phi|^2\right)^{-\frac{3}{2}}.\end{align}
Evaluating \eqref{derx1} in $0$, it follows $\frac{\partial^2 u_j}{\partial x_n \partial x_1}(0)=0$. In conclusion, the first and second derivatives of $u_j$ and $u_j^{\lambda_c}$ agree at $0$. We can apply Lemma \eqref{hopf} to the function $w=u_j^{\lambda_c}-u_j$ in $\Sigma'(\lambda_c)$ at $0$. This yields
\begin{equation*}\frac{\partial w}{\partial s}(0)<0\quad \text{or}\quad \frac{\partial^2 w}{\partial s^2}(0)<0,\end{equation*}
contradicting the fact that both $u_j$ and $u_j^{\lambda_c}$ have the same first and second partial derivatives at $\overline{y}$.
If $\frac{\partial u_i}{\partial x_n}(0)>0$ for $1\leq i\leq m$, in order to study the values $\frac{\partial^2 u_i}{\partial x_1\partial x_n}(0)$, we take into account the second boundary condition in \eqref{neum}, that is
\begin{equation}\label{diri4}
F\left(\frac{\partial u_1}{\partial \nu(x)},..., \frac{\partial u_m}{\partial \nu(x)}\right)=c.
\end{equation}
Differentiating \eqref{diri4} with respect to $x_1$ and evaluating in $0$, we obtain
\begin{align}\label{neum5}
\sum_{i=1}^m\frac{\partial F}{\partial x_i}\left(\frac{\partial u_1}{\partial x_n}(0),..., \frac{\partial u_m}{\partial x_n}(0)\right)\frac{\partial^2 u_i}{\partial x_n\partial x_1}(0)=0.
\end{align}
The sign of each of the terms on the left hand side of \eqref{neum5} is known. The partial derivatives $\frac{\partial u_i}{\partial x_1}$ restricted to $T_{\lambda_c}\cap\overline{\Omega}$ have a maximum at $0$, because of the first inequality in \eqref{in1}. From Schwarz's Theorem, it follows that $\frac{\partial^2 u_i}{\partial x_n \partial x_1}(0)\leq 0$. Therefore, equation \eqref{neum5} is equivalent to the following system of equations
\begin{equation*}\frac{\partial F}{\partial x_i}\left(\frac{\partial u_1}{\partial x_n}(0),..., \frac{\partial u_m}{\partial x_n}(0)\right)\frac{\partial^2 u_i}{\partial x_n\partial x_1}(0)=0\quad \forall i=1,..., m.\end{equation*}
By hypothesis, there exists an index $j$ such that $\frac{\partial F}{\partial x_j}>0$ in $(0, +\infty)^m$. Hence, $\frac{\partial^2 u_j}{\partial x_n \partial x_1}(0)=0$ and then the first and second derivatives of $u_j$ and $u_j^{\lambda_c}$ agree at $0$. We obtain a contradiction arguing as above.
\end{proof}
\begin{oss} The previous result still holds if functions $u_i$ are solutions to equations of the type described in Remark \ref{nonlin}, provided the function $F(x, z, p_k, r_{i,j})$ is independent of variables $r_{1,\alpha}$ for $\alpha>1$. This additional condition is necessary in order to verify hypothesis \eqref{cond1} in Lemma \ref{hopf}.
\end{oss}
\begin{oss} The monotonicity assumptions on function $F$ cannot be dropped as the following examples show.\\
By removing hypothesis $iii)$, we can set $F(x,y)=\begin{cases}
\frac{x}{y} & \mbox{if } y>0,\\ 
0 & \mbox{if } y=0,\\ 
\end{cases}$. Let $u_{\Omega}$ be the torsion function of $\Omega$ and $v=cu_{\Omega}$, for some constant $c$. Then the following system is satisfied
\begin{equation*}
\begin{cases}
-\Delta u_{\Omega}=1 & \mbox{in } \Omega,\\ 
-\Delta v=c & \mbox{in } \Omega,\\ 
u_{\Omega}=v=0 & \mbox{on } \partial\Omega,\\
F\left(\frac{\partial v}{\partial \nu},\frac{\partial u_{\Omega}}{\partial \nu}\right)=c& \mbox{on } \partial\Omega,\\
\end{cases}
\end{equation*}
for every smooth and bounded open set $\Omega\subset\mathbb{R}^n$. Similarly, we can set $u_{\Omega}$ and $v$ as above with $c\not=0$ and $F(x,y)=x^2-c^2y^2$, then we have
\begin{equation*}
\begin{cases}
-\Delta u_{\Omega}=1 & \mbox{in } \Omega,\\ 
-\Delta v=c & \mbox{in } \Omega,\\ 
u=v=0 & \mbox{on } \partial\Omega,\\
F\left(\frac{\partial v}{\partial \nu},\frac{\partial u_{\Omega}}{\partial \nu}\right)=0& \mbox{on } \partial\Omega,\\
\end{cases}
\end{equation*}
for every smooth and bounded open set $\Omega\subset\mathbb{R}^n$. \\
By removing hypothesis $iv)$, we can consider $F\equiv c$ identically constant, then the boundary value problem \eqref{neum} is not overdetermined anymore and the thesis becomes false in general. 
\end{oss}
\noindent In the end, we prove Corollary \ref{torlambda} by means of Theorem \ref{main} .

\begin{proof}[Proof of Corollary \ref{torlambda}]
Standard regularity results \cite[Theorem 6.8]{gilbtrud}, \cite[Theorem 1.2.12]{Henrot1} ensures that $u_{\Omega}, v_{\Omega}\in C^{2,\gamma}(\overline{\Omega})$. Since $\Omega$ is connected, \cite[Theorem 11.5.4]{jost} assures that the first eigenvalue $\lambda_1(\Omega)$ is simple and $v_{\Omega}$ can be chosen positive in $\Omega$. An application of the strong maximum principle shows that $u_{\Omega}>0$ in $\Omega$. By applying Theorems \ref{torsion}, \ref{firsteigen} for $\theta\in W^{1, \infty}(\mathbb{R}^n,\mathbb{R}^n)$ satisfying \eqref{volpres}, we get
\begin{equation}\label{d}
dJ(\Omega, \theta)=\int_{\partial \Omega}\left[\alpha T(\Omega)^{\alpha-1}|\nabla u_{\Omega}|^2-\beta\lambda_1^{\beta-1}(\Omega)|\nabla v_{\Omega}|^2\right]\theta\cdot n\,d\,\mathcal{H}^{n-1}=0, 
\end{equation}
 Since $\partial\Omega$ is smooth, condition \eqref{volpres} can be rewritten equivalently as
\begin{equation*}
    \int_{\Omega}\text{div }\theta\,dx=\int_{\partial\Omega}\theta\cdot n\,d\,\mathcal{H}^{n-1}.
\end{equation*}
Therefore, \eqref{d} implies
\begin{equation*}
    \alpha T(\Omega)^{\alpha-1}|\nabla u_{\Omega}|^2-\beta\lambda_1^{\beta-1}(\Omega)|\nabla v_{\Omega}|^2=c\quad \text{on }\partial\Omega.
\end{equation*}
If we set $F(x, y)=\alpha T(\Omega)^{\alpha-1}x^2-\beta\lambda_1^{\beta-1}(\Omega)y^2$, then Theorem \ref{main} can be applied and the thesis follows. 
\end{proof}
\subsection*{Acknowledgements}
The authors were partially supported by Gruppo Nazionale per l’Analisi Matematica, la Probabilità e le loro Applicazioni
(GNAMPA) of Istituto Nazionale di Alta Matematica (INdAM).   
\bibliographystyle{plain}
\bibliography{Bibliografia}

\begin{thebibliography}{10}

\bibitem{alexandrov}
A.~D. Alexandrov.
\newblock Uniqueness theorems for surfaces in the large. {V}.
\newblock {\em Amer. Math. Soc. Transl. (2)}, 21:412--416, 1962.

\bibitem{Bandle1}
C.~Bandle and A.~Wagner.
\newblock {\em Shape Optimization: Variations of Domains and Applications}.
\newblock Frontiers in Mathematics. De Gruyter, 2023.

\bibitem{bereniren}
H.~Berestycki and L.~Nirenberg.
\newblock On the method of moving planes and the sliding method.
\newblock {\em Bol. Soc. Brasil. Mat. (N.S.)}, 22(1):1--37, 1991.

\bibitem{bnst}
B.~Brandolini, C.~Nitsch, P.~Salani, and C.~Trombetti.
\newblock Serrin-{T}ype {O}verdetermined {P}roblems: an {A}lternative {P}roof.
\newblock {\em Arch. Rational Mech. Anal.}, 190:267–280, 2008.

\bibitem{Bmmtv}
G.~Buttazzo, F.~P. Maiale, D.~Mazzoleni, G.~Tortone, and B.~Velichkov.
\newblock Regularity of the optimal sets for a class of integral shape
  functionals.
\newblock {\em Ar{X}iv preprint}, 2022.

\bibitem{Chenr}
M.~Choulli and A.~Henrot.
\newblock Use of the {D}omain {D}erivative to {P}rove {S}ymmetry {R}esults in
  {P}artial {D}ifferential {E}quations.
\newblock {\em Mathematische Nachrichten}, 192:91--103, 1998.

\bibitem{ciron}
G.~Ciraolo and A.~Roncoroni.
\newblock The method of moving planes: a quantitative approach.
\newblock {\em Bruno Pini Math. Anal. Semin.}, 9:41–77, 2018.

\bibitem{DP}
L.~Damascelli and F.~Pacella.
\newblock Monotonicity and symmetry results for {$p$}-{L}aplace equations and
  applications.
\newblock {\em Adv. Differential Equations}, 5(7-9):1179--1200, 2000.

\bibitem{DS}
L.~Damascelli and B.~Sciunzi.
\newblock Regularity, monotonicity and symmetry of positive solutions of
  {$m$}-{L}aplace equations.
\newblock {\em J. Differential Equations}, 206(2):483--515, 2004.

\bibitem{dancer}
E.N. Dancer.
\newblock Some notes on the method of moving planes.
\newblock {\em Bulletin of the Australian Mathematical Society.}, 46:425--434,
  1992.

\bibitem{fk}
A.~Farina and B.~Kawohl.
\newblock Remarks on an overdetermined boundary value problem.
\newblock {\em Calc. Var.}, 31:351–357, 2008.

\bibitem{GNN}
B.~Gidas, W.~M. Ni, and L.~Nirenberg.
\newblock Symmetry and related properties via the maximum principle.
\newblock {\em Comm. Math. Phys.}, 68(3):209--243, 1979.

\bibitem{gilbtrud}
D.~Gilbarg and N.S. Trudinger.
\newblock {\em Elliptic Partial Differential Equations of Second Order}.
\newblock Classics in Mathematics. Springer Berlin Heidelberg, 2001.

\bibitem{Henrot1}
A.~Henrot.
\newblock {\em Extremum Problems for Eigenvalues of Elliptic Operators}.
\newblock Frontiers in Mathematics. Birkhäuser Basel, 2006.

\bibitem{Henrot2}
A.~Henrot and M.~Pierre.
\newblock {\em Shape Variation and Optimization: A Geometrical Analysis}.
\newblock EMS Tracts in Mathematics. European Mathematical Society, 2018.

\bibitem{jost}
J.~Jost.
\newblock {\em Partial Differential Equations}.
\newblock Graduate Texts in Mathematics. Springer New York, NY, 2013.

\bibitem{NT}
C.~Nitsch and C.~Trombetti.
\newblock The classical overdetermined serrin problem.
\newblock {\em Complex Variables and Elliptic Equations}, 63(7-8):1107--1122,
  2018.

\bibitem{Serrin}
J.~Serrin.
\newblock A symmetry problem in potential theory.
\newblock {\em Arch. Rational Mech. Anal.}, 43:304--318, 1971.

\bibitem{weinberger}
H.F. Weinberger.
\newblock Remark on the preceding paper of {S}errin.
\newblock {\em Arch. Rational Mech. Anal.}, 43:319--320, 1971.

\end{thebibliography}

\end{document}